\documentclass[oneside]{amsart}
\usepackage[T1]{fontenc}
\usepackage[utf8]{inputenc}
\usepackage[english]{babel}
\usepackage{refstyle}
\usepackage{amsmath}
\usepackage{amsthm}
\usepackage{amssymb}
\usepackage{listings}
\usepackage[skins,listings,breakable]{tcolorbox}
\usepackage{xcolor}
\usepackage{tikz}
\usepackage{pgfplots}
\usepackage{pdfpages}
\usepackage{hyperref}

\newtheorem*{thm*}{\protect\theoremname}
\newtheorem{thm}{\protect\theoremname}
\newtheorem{conjecture}{Conjecture}
\newtheorem{lem}[thm]{\protect\lemmaname}
\newtheorem{cor}[thm]{\protect\corollaryname}
\newtheorem{prop}[thm]{\protect\propositionname}

\theoremstyle{definition}
\newtheorem{defn}[thm]{\protect\definitionname}

\theoremstyle{remark}
\newtheorem{rem}[thm]{\protect\remarkname}
\providecommand{\corollaryname}{Corollary}
\providecommand{\definitionname}{Definition}
\providecommand{\examplename}{Example}
\providecommand{\lemmaname}{Lemma}
\providecommand{\propositionname}{Proposition}
\providecommand{\remarkname}{Remark}
\providecommand{\theoremname}{Theorem}
\providecommand{\claimname}{Claim}

\DeclareMathOperator{\NP}{NP}

\DeclareMathOperator{\lcm}{lcm}

\lstdefinelanguage{Magma}%
  {%
   otherkeywords={:=,+:=,-:=,*:=},%
   procnamekeys={function,func,intrinsic,procedure,proc},%
   morekeywords={true,false},%
   morekeywords=[2]{adj,and,cat,cmpeq,cmpne,diff,div,eq,ge,gt,in,is,join,le,lt,%
          meet,mod,ne,notadj,notin,notsubset,or,sdiff,subset,xor},%
   morekeywords=[3]{assigned,break,by,case,catch,continue,declare,default,%
          delete,do,elif,else,end,eval,exists,exit,for,forall,fprintf,if,local,%
          not,print,printf,quit,random,read,readi,repeat,restore,save,select,%
          then,time,to,try,until,vprint,vprintf,vtime,when,where,while},%
   morekeywords=[4]{clear,forward,freeze,iload,import,load},%
   morekeywords=[5]{assert,assert2,assert3,error,require,requirege,requirerange},%
   morekeywords=[6]{car,comp,cop,elt,ext,frac,hom,ideal,iso,lideal,loc,map,%
          ncl,pmap,quo,rec,recformat,rep,rideal,sub},%
   morekeywords=[7]{AbelianGroup,AdditiveCode,AffineAlgebra,Algebra,%
          AssociativeAlgebra,Character,CliffordAlgebra,Design,Digraph,%
          ExtensionField,FPAlgebra,FiniteAffinePlane,FiniteProjectivePlane,%
          Graph,Group,GroupAlgebra,IncidenceStructure,LieAlgebra,LinearCode,%
          LinearSpace,MatrixAlgebra,MatrixGroup,MatrixRing,Monoid,%
          MultiDigraph,MultiGraph,NearLinearSpace,Network,PartialMap,%
          PermutationGroup,PolycyclicGroup,QuaternionAlgebra,Semigroup,%
          ZModule},%
   morekeywords={[8]function,func,intrinsic,procedure,proc,return},%
      sensitive,%
      morecomment=[l]//,%
      morecomment=[s]{/*}{*/},%
      morecomment=[s]{\{}{\}},%
      morestring=[b]"%
  }[keywords,procnames,comments,strings]%
  
\newcommand{\QQ}{\mathbb{Q}}
\newcommand{\ZZ}{\mathbb{Z}}
\newcommand{\NN}{\mathbb{N}}

\begin{document}

\title[]{On the Galois Theory of Generalized Laguerre Polynomials and Trimmed Exponential}
\author{Lior Bary-Soroker}
\author{Or Ben-Porath}
\maketitle

\begin{abstract}
Inspired by the work of Schur on the Taylor series of the exponential and Laguerre polynomials, we study the Galois theory of trimmed exponentials $f_{n,n+k}=\sum_{i=0}^{k} \frac{x^{i}}{(n+i)!}$ and of the generalized Laguerre polynomials 
$L^{(n)}_k$ of degree $k$. 
We show that if $n$ is chosen uniformly from $\{1,\ldots, x\}$, then, asymptotically almost surely,  for all $k\leq x^{o(1)}$ the Galois groups of $f_{n,n+k}$ and of $L_{k}^{(n)}$ are the full symmetric group $S_k$.
\end{abstract}

\section{Introduction}
To a polynomial $f(x) = \sum_{j=0}^k b_j x^j$ with complex roots $\alpha_1,\ldots, \alpha_k$ we attach the Galois group $G_f$, which is defined as the automorphism group of $\QQ(\alpha_1,\ldots, \alpha_k)/\QQ$. It acts faithfully on the roots, and this action induces an embedding $G_f\leq S_k$. Already Galois showed that arithmetic properties of the $\alpha_i$-s are encoded in $G_f$. Therefore, in applications it is often times crucial to compute the Galois group of a specific polynomial or a family of polynomials. 

We call the problem of computing the Galois group of a polynomial or a family of polynomials \emph{the direct Galois problem}. 
Although there are algorithms to compute the Galois group of a given polynomial (see the survey paper \cite{hulpke1999techniques}), it is typically a challange  to compute the Galois group of a specific family of polynomials. For example, consider the family of the generalized Fibonacci polynomials $x^k-x^{k-1} -\cdots -1$. The construction of the Arnoux-Yoccoz surfaces \cite{ArnouxYoccoz81} uses the irreducibility of these polynomials, and further studies, cf.\ \cite{hooper2018rel}, need the more subtle information that the field $E/\QQ$ generated by a root has no totally real subextensions. If $k$ is even, the Galois group is $S_k$ \cite{martin2004galois} hence $E$ is minimal. If $k$ is odd, the Galois group is also conjectured to be $S_k$, but it is still open in general. Nevertheless, the required property for the application is proved by the first author, Shusterman and Zannier \cite[Appendix]{hooper2018rel}. 

The direct Galois problem was studied from a probabilistic approach. Roughly speaking, the Galois group tends to be the "largest possible".  Maybe the first result in this spirit goes back to van-der Waerden \cite{Waerden} who proved that the probability of a random uniform  polynomial of degree $k$ and integral coefficients $|a_i|\leq H$ to have Galois group $S_k$ tends to $1$ as $H$ tends to infinity. See \cite{Gallagher, Dietmann} for better bounds on the error term. Recently, the case when $H$ is fixed (e.g.\ $H\geq 35$) and the degree tends to infinity was established \cite{bary2020irreducible, breuillard2019irreducibility, bary2020irreducibility}.

One can approach the direct Galois problem in a more explicit manner and there is a line of works on polynomials coming from analytic functions. 
Let $e_k(x) = \sum_{j=0}^k \frac{x^k}{j!}$ be the Taylor polynomial of degree $k$ of the exponential function. 
In 1930, Schur \cite{schur1930gleichungen} computed the Galois group of $e_k$:
\[
    G_{e_k}=
    \begin{cases}
        A_{k} & k\equiv0(4), \\
        S_{k} & \textnormal{otherwise}.
    \end{cases}
\]
One may trim the Taylor series of $e^x$ on both sides. After normalization, the trimmed exponential has the form
\[
    f_{n,m} (x) = \sum_{j=n}^m \frac{x^{j-n}}{j!} , \qquad m>n.
\]
When $n>1$ it seems that the discriminant is never a square, hence the Galois group should not be the alternating group. 
Therefore one may pose the following 
\begin{conjecture}
    For every $m>n\geq 1$ the polynomial $f_{n,m}$ is irreducible and moreover
    \[
        G_{f_{n,m}} = S_{m-n}.
    \]
\end{conjecture}
Numerical experiments suggest that this conjecture is valid. 
The goal of this paper is to combine the probabilistic approach with the explicit approach to get partial results and thus to provide  evidence to this conjecture. Our main result is that when $m$ is chosen at random, then the conjecture holds with high probability for all $n$ which is relatively close to $m$:
\begin{thm}
\label{MAIN-thm-Sk-probability}
Choose $m$ uniformly at random in $\{ 1,\ldots,x\}$. Let $k_{\max}=k_{\max}(x)=x^{o(1)}$. Then
\[
    \displaystyle{\lim_{x\to\infty}}\mathbb{P}\left(\forall\  8\leq k\leq k_{\max}: G_{f_{m-k,m}}=S_{k}\right)=1.
\]
\end{thm}

Our approach fits in a broader point of view.
The \emph{generalized Laguerre polynomial} $L_k^{(\alpha)}$ of degree $k$ and parameter $\alpha$ is defined as the polynomial solution of the ordinary differential equation
\[
    xy''+(\alpha+1-x)y'+ky=0.
\]
It has the closed form
\[
    L_k^{(\alpha)}(x)= \sum_{j=0}^{k}\frac{(k+\alpha)(k-1+\alpha)\cdots(j+1+\alpha)}{(k-j)!j!}(-x)^{j}.
\]
The family of generalized Laguerre polynomials includes the exponential Taylor polynomials  $e_k(x)=L_k^{(-k-1)}$ and the Laguerre polynomials when $\alpha =0$. When $\alpha = -2k-1$,  we  get the reverse Bessel polynomials and when $\alpha =\pm 1/2$ the Hermite polynomials (up to normalization).

In 1929, Schur \cite{schur1929einige}  proved the irreducibility of $e_k=L_k^{(-k-1)}$ and $L_k^{(0)}$ and in \cite{schur1930gleichungen} he computed the corresponding Galois groups: for $e_k = L_k^{(-k-1)}$ it was discussed above, and $L_k^{(0)}$ has the full symmetric group. 

A year later, he \cite{schur1931affektlose}  proved that $L_{k}^{(1)}$ 
has Galois group
\[
G_{L_k^{(1)}}=\begin{cases}
A_{k} & k\textnormal{ is odd or }k+1\textnormal{ is an odd square}, \\
S_{k} & \textnormal{otherwise}
\end{cases}
\]
and that $L_k^{(\pm 1/2)}$ has Galois group $S_k$ for $k>12$. Presumably Schur's motivation was to give explicit examples of polynomials with Galois group $A_k$, following Hilbert's realization of $A_k$ using specialization methods.

Following Schur, $L_k^{(\alpha)}$ was studied in many ranges and special values of $\alpha$. In 1978, Grosswald \cite{grosswald1978bessel} showed that the Galois group of the reverse Bessel polynomials $L_k^{(-2k-1)}$ is the symmetric group, under the condition that they are irreducible. In 2002, Filaseta and Trifonov \cite{filaseta2002irreducibility2} proved the irreducibility. 

Also in 2002, Filaseta and Lam \cite{filaseta2002irreducibility1} showed that $L_k^{(\alpha)}$ is irreducible for all but finitely many $k$-s, for any choice of rational number $\alpha$ that is not a negative integer. In 2005, Hajir \cite{hajir2005galois} showed that in this case 
$G_{L_k^{(\alpha)}}\geq A_{k}$. There are further recent results on these families, see for example \cite{filaseta2012laguerre,hajir2009algebraic}.

In these results $k$ is large w.r.t.\ $\alpha$. We prove a result for small degrees. Our result is valid for almost all $\alpha$ in the probabilistic sense and for all small $k$: 

\begin{thm}\label{Main-thm-Laguerre}
    Choose $\alpha\in \{1,\ldots x\}$ uniformly at random and let $k_{\max} = x^{o(1)}$. Then
    \[
        \lim_{x\to \infty}\mathbb{P}(\forall 8\leq k\leq k_{\max}, \ L_{k}^{(\alpha)} \textnormal{ is irreducible with Galois group $S_k$}) =1.
    \]
\end{thm}

In fact our results hold for a larger family of polynomials that unifies Theorems~\ref{MAIN-thm-Sk-probability} and \ref{Main-thm-Laguerre}, see Section~\ref{sec-methods}.

Small degrees $k\leq 7$ are discussed in Section~\ref{SubSec_small_k}.

\subsection*{Acknowledgments}
We thank Michael Filaseta for the historical remark on the motivation of Schur. We are also grateful to Dimitris Koukoulopoulos on helpful suggestions regarding smooth numbers.

The authors were supported by a grant of the Israel Science Foundation, no.\  702/19.

This work is based on a master thesis done by OBP under the supervision LBS.

\section{Proof of Theorems~\ref{MAIN-thm-Sk-probability} and \ref{Main-thm-Laguerre}}\label{sec-methods}
In this section we discuss the family of polynomials $[f]$ coming from a polynomial $f$,  we state a few technical propositions, and we deduce Theorems~\ref{MAIN-thm-Sk-probability} and \ref{Main-thm-Laguerre} from the propositions. The proofs of the propositions appear in Section~\ref{sec-proof-propositions}. In that section we will also give bounds on the convergence rates.

For a polynomial $f=\sum_{j=0}^{k} b_j x^j$ we define
\begin{equation}
\label{eq:def_fam}
    [f]=\Big\{\sum_{j=0}^{k}a_j b_j x^j:a_j \in \mathbb{Z} \mbox{ and $a_0a_k\neq 0$} \Big\}.
\end{equation}

Schur \cite{schur1929einige} showed that 
$f\in [L_k^{(-k-1)}]$ is irreducible if
$|a_{0}|=|a_{k}|=1$. 
Filaseta \cite{filaseta1996generalization} relaxed Schur's condition to 
$|a_{0}|=1$, $0<|a_{k}|<k$,  
and 
$(k,a_{k})\neq(6,\pm5),(7,\pm10)$.
Filaseta and Lam \cite{filaseta2002irreducibility1} showed that for $\alpha$ a rational number that is not a negative integer and $|a_{0}|=|a_{k}|=1$, $f\in [L_k^{(\alpha)}]$ is irreducible for all but finitely many $k$.

Since
\[
    L_k^{(\alpha)}(x)= \sum_{j=0}^{k}{k \choose j} \frac{(\alpha+k)!}{k!} \frac{1}{(j+\alpha)!}(-x)^{j}.
\]
and $a_j:=(-1)^j {k \choose j} \frac{(\alpha+k)!}{k!}$ is integral, we have 
$$
    [L_k^{(\alpha)}]\subset [f_{\alpha,\alpha+k}].
$$
This motivated Filaseta, Finche, and Leidy \cite{filaseta2008tn} to suggest to study the families $[f_{\alpha,\alpha+k}]$ as a form of generalization of  $[L_k^{(\alpha)}]$.

Throughout this section  $m$ is a uniformly chosen random integer in the interval $[1,x]$. We recall that every $f\in \Big[\sum_{j=0}^k b_jx^j\Big]$ comes with a tuple $(a_0,\ldots, a_k)$, as in (\ref{eq:def_fam}). 
We define three families of events. For a positive integer $k$ and for primes $p,q$ (not necessarily distinct) let 
\[
    \begin{split}
        A_k^p &= \{ \forall f\in [f_{m-k,m}],\ (p\nmid a_0a_k \Rightarrow f \mbox{ is irreducible})\},\\
        B_k^p &= \{\forall f\in [f_{m-k,m}],\ (p\nmid a_0a_{2\lfloor \frac{k}{2}\rfloor} \Rightarrow G_f\neq A_k)\},\\
        C_{k}^{q,p} & = \{ \exists j_1,j_2=j_1+p, \forall f\in [f_{m-k,m}], (q\nmid a_{j_1}a_{j_2}\Rightarrow G_f\geq A_k)\}.
    \end{split}
\]

The technical propositions compute the  probabilities of these events in certain ranges. 

For $A_k^p$ and $B_k^p$ we need that $k\leq x^{o(1)}$:
\begin{prop}
\label{MAIN-cor-irreducibility-probability}
Let $k_{\max}=x^{o(1)}$. Then 
\[
\displaystyle{\lim_{x\to\infty}}\mathbb{P}\left( \displaystyle{\bigcup _{p>k_{\max}}\bigcap _{k=1}^{k_{\max}}}A_k^p\right)=1.
\]
\end{prop}

\begin{prop}
\label{MAIN-prop-not-Ak-probability}

Let $k_{\max} = x^{o(1)}$. Then
\[
\displaystyle{\lim_{x\to\infty}}\mathbb{P}\left(\displaystyle{\bigcup _{p>k_{\max}}\bigcap _{k=1}^{k_{\max}}}B_k^p\right)=1.
\]
\end{prop}

For $C_k^{q,p}$ we get a wider range of $k$: 

\begin{prop}
\label{MAIN-prop-Ak-probability}
Let $k_{\max} = x^{1/7-\epsilon}$. Then 
\[
\displaystyle{\lim_{x\to\infty}}\mathbb{P}\left(\displaystyle{\bigcap _{k=8}^{k_{\max}}\bigcup _{q > k > p > \frac{k}{2}}}C_k^{q,p}\cup C_k^{p,p}\right)=1.
\]
\end{prop}

We are now ready to prove the theorems:
\begin{proof}[Proof of Theorem~\ref{MAIN-thm-Sk-probability}]
    By the above propositions, with probability $1+o_{x\to \infty} (1)$, for every $k=8,\ldots, k_{\max}$ there exist $p_1,p_2,p_3,p_4$ such that $p_1,p_2>k>p_3>k/2$ and either $p_4>k$ or $p_4=p_3$, such that $A_k^{p_1}$, $B_{k}^{p_2}$, and $C_k^{p_3,p_4}$ hold. 
    Since $f_{m-k,m}\in [f_{m-k,m}]$ with $a_j=1$ for all $j$, we conclude that $f_{m-k,m}$ is irreducible, with Galois group $G_f=S_k$, as needed. 
\end{proof}

\begin{proof}[Proof of Theorem~\ref{Main-thm-Laguerre}]
    We start exactly the same as in the proof of Theorem~\ref{MAIN-thm-Sk-probability}. Now, the polynomial 
    \[
        f = \frac{k!}{m!} L_k^{(m-k)}
    \]
    satisfies $f\in [f_{m-k,k}]$ with $a_j = (-1)^j\binom{k}{j}$. In particular,
    $p_1,p_2\nmid a_j$ for all $j$, so  with probability $1+o_{x\to \infty}(1)$, the polynomial $f$ is irreducible and $G_f\neq A_k$. If $p_4>k$, then we also have that $G_f\geq A_k$, and so $G_f=S_k$ and the proof is done. 
    
    To this end assume that $k>p_3=p_4>k/2$.  To finish the proof, it suffices to show that $p_3\nmid a_{j_1}a_{j_2}$ for every $j_1, j_2=j_1+p_3$. 
    We have 
    \[
        v_{p_3}(a_j)=1-v_{p_3}(j!(k-j)!)=
        \begin{cases}
            1 & j \in (k-p_3,p_3), \\
            0 & \textnormal{otherwise}.
        \end{cases}
    \]
Both $j_1=j_2-p_3 \leq k-p_3$ and $j_2=j_1+p_3 \geq p_3$ lie outside the interval $(k-p_3,p_3)$, therefore $p_3 \nmid a_{j_1}a_{j_2}$, as required.
\end{proof}

\numberwithin{thm}{section}

\section{Newton Polygons}

\subsection{Setting}
Let $\QQ_p\subseteq K \subseteq L$ be a tower of finite extensions. We denote by 
$e=e(L/K)$ 
the ramification index and by $v$ the unique extension of the $p$-adic valuation to $L$, normalized so that 
$v(K^{\times})=\ZZ$.

\begin{defn}
\label{NP-def}
Let $f\in K[x]$
be of degree $k$, and assume that $f(0)\neq 0$. 
Write $f=\sum_{j=0}^{k}b_{j}x^{j}$
and consider the set of points
\[
B(f)=B^{(p)}(f)=
\{ (j,v(b_{j})):0\leq j\leq k\}\subset\ZZ\times(\ZZ\cup\{ \infty\}).
\]
We recall the definition of the \emph{Newton polygon} of $f$ as the boundary of the lower convex hull of $B(f)$; i.e., the lowest possible piecewise linear function that includes 
$(0,v(b_{0})),(k,v(b_{k}))$
and all of its vertices are in 
$B(f)$. 
We denote it by 
$\NP(f)$ 
or 
$\NP^{(p)}(f)$
if we wish to emphasize the valuation.

\label{NP-def-lattice-points}
We denote the set of lattice points on $\NP(f)$ by
\[
C(f)=C^{(p)}(f)=
\NP(f)\cap(\ZZ\times\ZZ).
\]
\end{defn}

\begin{rem}
\label{NP-rem-compactness}
By assumption $b_0,b_k$ are nonzero, hence
$v(b_{0}),b(f_{k})\in \ZZ$, hence $\NP(f)$ is compact.
\end{rem}

\begin{figure}[h]
    \centering
\fbox{\includegraphics{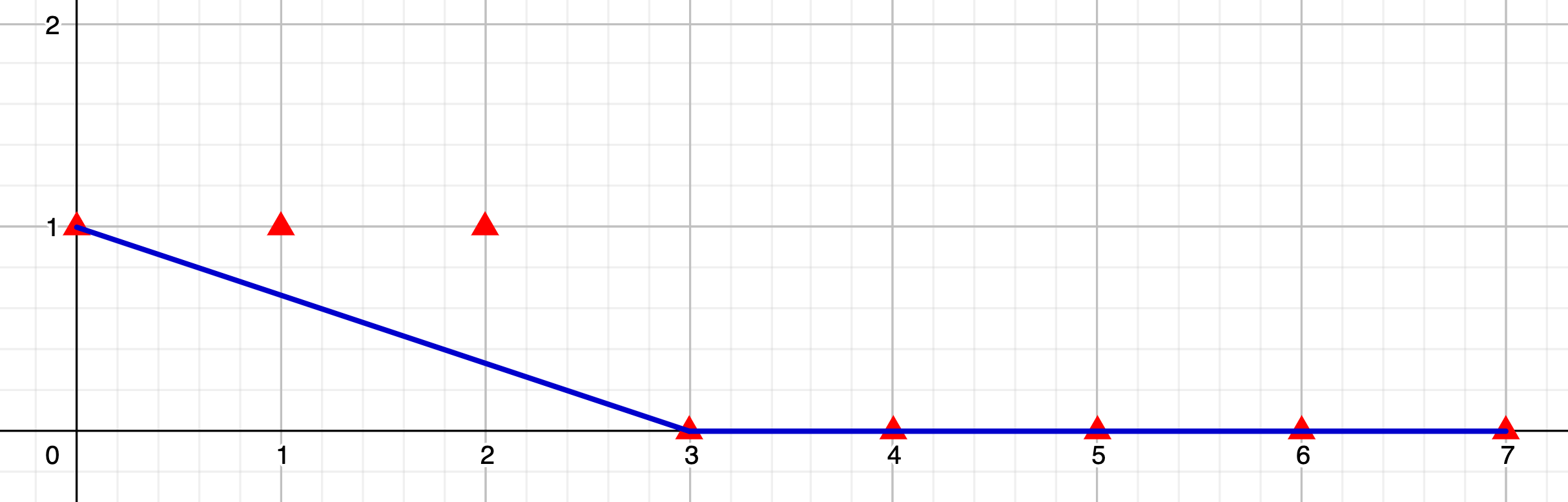}}
    \caption{$\NP^{(p)}(m! f_{m,n})$ with $n=1342340$, $m=1342347$, and $p=1342343$}
\end{figure}

\subsection{Basic Properties} By definition,  $\NP(f)$ is piecewise linear with increasing rational slopes. The set of vertices of $\NP(f)$ is a subset of $B(f)$, and usually a proper subset. 
\begin{defn}
We denote the sequence of slopes of $\NP(f)$
\label{NP-def-slope-sequence}
by
\[
\beta_{j}:=\NP(f)(j)-\NP(f)(j-1).
\]
\end{defn}

We collect below a few well known properties. See \cite{dumas1906quelques} for proofs.  

The slopes of $NP(f)$ give the valuation of the roots of $f$; hence relating the valuation of the coefficients with the valuation of the roots. This idea goes back to Newton, in the context of complex polynomials in two variables, see \cite{brieskorn2012plane}. More precisely:

\begin{lem}
\label{NP-lem-Newton}
Let $f\in K[x]$ be of degree $k$ with $f(0) \neq 0$, and $\alpha$ a root of $f$. Then for some $1 \leq j \leq k$ we have $v(\alpha)=-\beta_{j}$.
\end{lem}

 The next two lemmas relate $\NP(gh)$ with $\NP(g)$ and $\NP(h)$. 

\begin{lem}
\label{NP-lem-Dumas-product}
Let $f=gh$ where $g,h \in K[x]$ with $g(0),h(0) \neq 0$. Then the sequence of slopes of $f$ is the same as the combined sequence of slopes of $g$ and $h$, ordered to be increasing. 
\end{lem}

\begin{lem}
\label{NP-lem-Dumas-factorization}
Let $f \in K[x]$, with $f(0) \neq 0$. If $\NP(f)$ consists of $r$ segments of distinct slopes $m_1,\ldots, m_r$ and lengths $d_1,\ldots, d_r$, then $f$ has a factorization of the form $f={\displaystyle \prod_{i=1}^{r}}g_{i}$, where $g_i\in K[x]$, $\deg(g_{i})=d_{i}$, and $\NP(g_{i})$ is a line with slope $m_{i}$.
\end{lem}

Recall that a polynomial $f$ is Eisenstein at $p$, if $C^{(p)}(f) = \{(0,1), (k,0)\}$ and then $\NP(f)$ consists of one line with slope $-1/\deg f$. The Eisenstein criterion says that in this case $f$ is irreducible. 
More generally we have
\begin{lem}\label{NP-lem-Eisenstien-large-cycle}
    Let $f\in \QQ[x]$ and $p$ be a prime. Assume that $\NP^{(p)}(f)$ contains a segment whose end points are $(c, d)$ and $(c+t,d-s)$ such that $\gcd(s,t)=1$ and $p\nmid t$. Then the Galois group of $f$ over $\QQ_p$ contains an element having a $t$-cycle in its factorization to disjoint cycles. In particular $G_f$ contains such an element. 
\end{lem}

\begin{proof}
    By Lemma~\ref{NP-lem-Dumas-factorization}, $f=gh$, where $g,h\in \QQ_p[x]$, $\deg g=t$ and $\NP(g)$ is a line with slope $-\frac{s}{t}$. By Lemma~\ref{NP-lem-Newton}, a root $\alpha$ of $g$ has valuation $\frac{s}{t}$. 
    In particular, if $\alpha$ is a root of $g$, we get that $[\QQ_p(\alpha):\QQ_p]=t$, and this extension is totally ramified. Hence the inertia group acts transitively on the roots of $g$.
    On the other hand, since $p\nmid t$, the ramification group of the splitting field of $g$ is tame, hence the inertia group acts cyclically on the roots of $g$. Thus a lifting of a generator of the inertia group of $g$ to the splitting field of $f$ is the desired element. 
\end{proof}

\begin{lem}
\label{NP-lem-Tilde-1}
Let $f=\sum_{j=0}^{k} b_j x^j \in K[x]$ be of degree $k$, with $f(0) \neq 0$, and $g=\sum_{j=0}^{k} a_j b_j x^j\in [f]$. If 
$(c,d) \in C(f)$ and $v(a_c)=0$, then 
$(c,d) \in C(g)$.
\end{lem}

\begin{proof}
Regardless of the restriction on $a_{c}$, it is always true that 
\[
\NP(f) \leq \NP(g).
\]
By assumption on $a_{c}$
\[
\NP(g)(c) \leq \NP(f)(c) = d.
\]
It follows that $\NP(g)(c)=d$ as required.
\end{proof}

\begin{lem}
\label{NP-lem-Tilde-2}
Let $f=\sum_{j=0}^{k} b_j x^j \in K[x]$ be of degree $k$, with $f(0) \neq 0$, and $g=\sum_{j=0}^{k} a_j b_j x^j\in [f]$. If 
$\NP(f)$ contains a line of slope $-\frac{s}{t}$ from $(c,d)$ to $(c+t,d-s)$ and $v(a_{c}),v(a_{c+t})=0$, then $\NP(g)$ contains the same line from 
$(c,d)$ to $(c+t,d-s)$.
\end{lem}

\begin{proof}
By Lemma~\ref{NP-lem-Tilde-1}, $\NP(g)(c)=\NP(f)(c)=d$ and $\NP(g)(c+t) = \NP(c+t)=d-s$. On the other hand, we always have 
\[
\NP(f)\leq \NP(g).
\]
By convexity of $\NP(g)$ this means that $\NP(g)$ and $\NP(f)$ coincide on the interval $[c,c+t]$, as needed. 
\end{proof}

\section{Smooth numbers}
Recall a natural number $m \in \NN$ is called \emph{$k$-smooth} if every prime divisor $p \mid m$ is $\leq k$.
The number of all $k$-smooth numbers smaller than $x$ is denoted by
\[
\psi(x,k):=\# \{ m\leq x:m \textnormal{ is } k\textnormal{-smooth}\}.
\]

An exact asymptotic for $\psi(x,k)$ in various uniformity regimes is known, cf.\ \cite[Chapter 3]{norton1971numbers} for a survey of such results. We bring below an upper bound, having the advantages of being uniform in $x$ and $k$ and of having an easy proof of which we have learnt from  Koukoulopoulos. 

\begin{lem}
\label{SM-lem-basic}As $x$ tends to infinity we have
\[
\psi(x,k) \leq \frac{x}{\log x} (\log k+O(1)).
\]
\end{lem}

\begin{proof}
Let $S(x,k)$ denote the set of $k$-smooth numbers smaller than $x$, so that 
\[
\psi(x,k)=|S(x,k)|.
\] 
Consider 
\[
\displaystyle{\sum_{s \in S(x,k)}}\log s.
\]
On the one hand, by partial summation
\[
\displaystyle{\sum_{s \in S(x,k)}} \log s = 
\psi(x,k)\log x - \int_{1}^{x} \frac{\psi(u,k)}{u} du \geq
\psi(x,k)\log x - x.
\]
On the other hand, by substituting 
$\log s=\displaystyle{\sum _{\underset{p\textnormal{ prime}}{p^{r}\mid s}}}\log p$
we get
\[
\begin{aligned}
\displaystyle{\sum_{s \in S(x,k)}}\log s &=
\displaystyle{\sum_{s \in S(x,k)}}\ \displaystyle{\sum _{\underset{p\textnormal{ prime}}{p^{r} \mid s}}} \log p =
\displaystyle{\sum_{\underset{p\textnormal{ prime}}{p^{r} \leq x,p \leq k}}}\log p \displaystyle{\sum _{t \in S(\frac{x}{p^r},k)}}1 \\
&\leq \displaystyle{\sum_{\underset{p\textnormal{ prime}}{p^{r} \leq x,p \leq k}}} \frac{x \log p}{p^{r}}
\leq x (\log k + O(1)).
\end{aligned}
\]
(The last inequality follows from Mertens' first theorem $\sum_{p\leq k} \frac{\log p}{p}\leq \log k+2$.)
Combining the two inequalities and rearranging finishes the proof.
\end{proof}

\begin{prop}
\label{SM-prop-non-square-probability}
Let $m$ be a uniform random integer in $[1,x]$. Let $k=k(x)$. Then
\[
\mathbb{P}\left(\exists p>k:v_{p}\left(m\right)=1\right)=1-O\left(\frac{\log k}{\log x}+\frac{1}{k}\right)
\]
with absolute implied constant. 
\end{prop}

\begin{proof}
By Lemma \ref{SM-lem-basic},
\[
\mathbb{P}\left(\exists p>k:v_{p}(m)\geq 1\right)=1-O\left(\frac{\log k}{\log x}\right).
\]
On the other hand, 
\[
    \begin{aligned}
    \mathbb{P}\left(\exists p>k: v_p(m)>1\right) 
    &\leq \frac{|\{m\leq x: \exists p>k \textnormal{ s.t. } p^{2}\mid m\}| }{x} \\
    & \leq \frac{1}{x}\displaystyle{\sum_{\underset{p \textnormal{ prime}}{k<p\leq \sqrt{x}}}} \frac{x}{p^{2}}
    \leq \displaystyle{\sum_{\underset{p \textnormal{ prime}}{k<p}}} \frac{1}{p^{2}}
    \leq \displaystyle{\sum_{n=k}^{\infty}} \frac{1}{n^{2}}
    =O\left(\frac{1}{k}\right).
    \end{aligned}
\]
The result follows.
\end{proof}

We shall need the following refined version of Proposition~\ref{SM-prop-non-square-probability}.

\begin{cor}
\label{SM-cor-unbounded-non-square-probability}
Let $m$ be a uniform random integer in $[1,x]$. Let $k=k(x)$ and $t=t(x)$ be functions of $x$.
Then
\[
\mathbb{P}\left(\forall\ 0\leq i< t:\ \exists p_{i}>k \textnormal{ prime s.t. }v_{p_{i}}(m-i)=1\right)=1-O\left(\frac{t\log k}{\log x}+\frac{t}{k}\right).
\]
\end{cor}

\begin{proof}
If $t\geq \log x$, then the statement is trivial. Next assume that $t\leq 
\log x$, hence  $\log (x-i)\sim \log x$, as $x\to \infty$.
By Proposition~\ref{SM-prop-non-square-probability} and the union bound we have
\begin{multline*}
\mathbb{P}\left(\forall\ 0\leq i< t:\ \exists p_{i}>k \textnormal{ s.t. }v_{p_{i}}(m-i)=1\right)\geq 
1-\displaystyle{\sum_{i=0}^{t-1}}O\left(\frac{\log k}{\log (x-i)}+\frac{1}{k}\right)
\\=1-\displaystyle{\sum_{i=0}^{t-1}}O\left(\frac{\log k}{\log x}+\frac{1}{k}\right) 
=1-O\left(\frac{t\log k}{\log x}+\frac{t}{k}\right),
\end{multline*}
as needed.
\end{proof}

\section{Proof of the technical propositions}
\label{sec-proof-propositions}
As usual, and throughout this section, $m$ is a uniformly chosen integer in the interval $[1,x]$.

The following is a version of Proposition~\ref{MAIN-cor-irreducibility-probability} with explicit error term.
\begin{prop}\label{RE-thm-irreducibility-probability}
    Let $\kappa=\kappa(x)$ be a function of $x$. Let ${A_{k}^{p}}'$ be the event that for every $1\leq k\leq \kappa$ and every $f\in [f_{m-k,m}]$ with $p\nmid a_0a_k$ the Galois group $G_f$ contains a $k$-cycle. 
    Then
    \[
        \mathbb{P} \left(\bigcup_{p>\kappa}\bigcap_{k=1}^{\kappa} A_k^p\right)
        \geq 
        \mathbb{P} \left(\bigcup_{p>\kappa}\bigcap_{k=1}^{\kappa} {A_k^p}'\right) 
        =1-O\left(\frac{\log \kappa}{\log x}+\frac{1}{\kappa}\right).
    \]
\end{prop}

\begin{proof}
    Obviously ${A_k^p}'\subseteq A_k^p$, hence the left hand side inequality. 
    Proposition~\ref{SM-prop-non-square-probability} implies that with probability $1-O\Big(\frac{\log \kappa}{\log x}+\frac{1}{\kappa}\Big)$ there exists a prime $p>\kappa$ such that $v_p(m)=1$. Thus $p\nmid m-j$, for all $1\leq j\leq p-1$ and in particular for all $1\leq j\leq \kappa-1$. 
    Take $f\in [f_{m-k,m}]$ with $p\nmid a_0a_k$ and consider $g=m!f$. 
    \begin{figure}[h]
        \centering
        \fbox{\includegraphics[width=5cm]{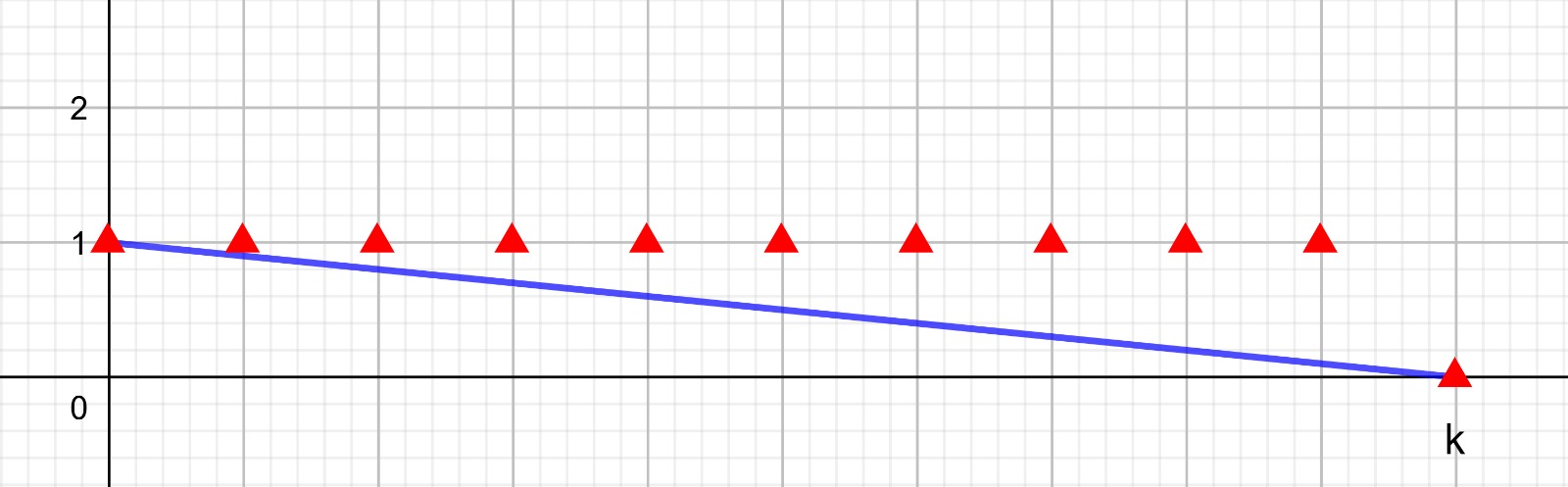}}
        \caption{$\NP(g)$}
        \label{fig:my_label}
    \end{figure}
    Then the Newton polygon consists of a single segment of slope $-1/k$. By Lemma~\ref{NP-lem-Eisenstien-large-cycle},  $G_g$ and hence also $G_f$ contain a $k$-cycle. 
\end{proof}

\begin{proof}[Proof of    Proposition~\ref{MAIN-cor-irreducibility-probability}]
    Apply Proposition~\ref{RE-thm-irreducibility-probability} with $\kappa = \max\{ k_{\max}, \log x\}$. In this case the error term tends to $0$ as $x$ tends to $\infty$. 
\end{proof}

Next we prove a version of Proposition~\ref{MAIN-prop-not-Ak-probability} with explicit error terms. 

\begin{prop}
    \label{prop-no-Ak-ET}
    Let $\kappa=\kappa(x)$ be a function of $x$. Let ${B_k^p}'$ be the event that for every $f\in [f_{m-k,m}]$ with $p\nmid a_0a_{2\lfloor k/2\rfloor} $ the Galois group $G_f$ contains a $2\lfloor k/2\rfloor$-cycle.  Then 
    \[
        \mathbb{P} \left(\bigcup_{p>\kappa} \bigcap_{k=1}^{\kappa} B_k^p \right)\geq     
        \mathbb{P} \left(\bigcup_{p>\kappa} \bigcap_{k=1}^{\kappa} {B_k^p}' \right) = 1 - O\left(\frac{\log \kappa}{\log x} + \frac{1}{\kappa}\right).
    \]
\end{prop}

The proof of Proposition~\ref{prop-no-Ak-ET} is similar to the proof of Proposition~\ref{RE-thm-irreducibility-probability}. We leave the details to the reader.

 \begin{proof}[Proof of Proposition~\ref{MAIN-prop-not-Ak-probability}]
     Applying Proposition~\ref{prop-no-Ak-ET} with $\kappa = \max\{ k_{\max}, \log x\}$ implies Proposition~\ref{MAIN-prop-not-Ak-probability}. 
 \end{proof}

\subsection{Large Galois Group}
For this part we need a deep result from analytic number theory, the existence of primes in short intervals: 
\begin{equation}\label{eq-exist-prime-SI}
    \exists y<p<y+y^{\rho}, \qquad \forall y\gg 1
\end{equation}
for some $\rho>0$. The smaller $\rho$ is the larger the range of uniformity in which our result holds. The state-of-the-art is $\rho=0.525$, due to  Baker, Harman and Pintz \cite{baker2001difference}.

We first prove that if $k$ grows with $x$, then the Galois group contains a $p$-cycle with $\frac{k}{2}<p<k-2$.

\begin{prop}\label{prop-largek-Ak}
    Let $\kappa_1\leq \kappa_2$ be functions of $x$ such that $\displaystyle\lim_{x\to \infty}\kappa_1=\infty$ and $\kappa_2^{7+\epsilon} = O(x)$, for some $0.425>\epsilon>0$. Let $\Gamma_{k}^p$ be the event that there exist two indices $j_1, j_2 = j_1 +p$ such that 
    for any $f\in [f_{m-k,m}]$ with $p\nmid a_{j_1}a_{j_2}$ we have $p\mid |G_{f}|$. Then 
    \[
        \mathbb{P}\left( \bigcap_{\kappa_1\leq k\leq  \kappa_2}\bigcup_{\frac{k}{2}<p<k-2} \Gamma_k^p\right) = 1-O(\kappa_1^{-\epsilon})
    \]
\end{prop}

\begin{proof}
    Fix $\kappa_1\leq k\leq \kappa_2$ and $\frac{k}{2}<p<\frac{k}{2}+h<k-2$, for some $h$. Let $A = \{ 0\leq i \leq k-p\}\subseteq \mathbb{Z}/p\mathbb{Z}$. So $m \pmod p\in A$ if and only if there are two multiples of $p$ in the interval $[n,m]$ (recall that $n=m-k$). 
    Since $m$ is a uniform integer in $[1,x]$, $m\pmod{p}$ is uniform in $\mathbb{Z}/p\mathbb{Z}$ up to an error term of $O(p/x)$, so we have 
    \[
    \begin{split}
        &\mathbb{P}(\exists  m_p'\neq m_p\in [n,m] : p\mid m_p',m_p) 
        \\
        &\qquad =\frac{|A|}{p} - O\Big(\frac{p}{x}\Big) = 1 - O\Big(\frac{h}{p}+\frac{p}{x}\Big) = 1 - O\Big(\frac{h}{k}+\frac{k}{x}\Big)=1-O\Big(\frac{h}{k}\Big).        
    \end{split}
    \]
    Let $B = \{0\leq i \leq p-1 \}\subseteq \mathbb{Z}/p^2 \mathbb{Z}$. So $m\pmod{p^2}\not\in B$ if and only if $v(m_p)=1$, where $m_p = p\cdot \lfloor m/p\rfloor$ is the largest multiple of $p$ in $[n,m]$. As before $m\pmod{p^2}$ is uniform in $\mathbb{Z}/p^2 \mathbb{Z}$ up to an error term of $O(p^2/x)$. Thus we get that 
    \[
        \mathbb{P} (v_p(m_p)=1) = 1-\frac{|B|}{p^2}-O\Big( \frac{p^2}{x}\Big) = 1- O\Big(\frac{1}{k}+\frac{k^2}{x}\Big)=1-O\Big( \frac{1}{k} \Big).
    \]
    We apply this to three primes $\frac k2<p_1, p_2, p_3<\frac k2+h$. Put $P=(p_1p_2p_3)^2$. Using the Chinese Remainder Theorem and that $m \pmod P$ is uniform in $\mathbb{Z}/P\mathbb{Z}$ up to an error term of $O(P/x)=O(k^6/x)$, we get that 
    \begin{equation}\label{eq-threeprimesprob}
    \begin{split}
        &\mathbb{P}(\exists i: \exists m_{p_i}'<m_{p_i}\in [n,m], p_i\mid m_{p_i}',m_{p_i} \mbox{ and } v_{p_i}(m_{p_i})=1) \\ &\qquad\qquad= 1 -O\Big( \frac{h^3}{k^3}+\frac{k^6}{x}\Big).        
    \end{split}
    \end{equation}
    
    Next we note that if $[n,m]$ contains two multiples $m_p'<m_p$ of $p$ and if $v_p(m_p)=1$, then $\Gamma_k^p$ holds, with $j_1=  m_p'$ and $j_2=m_p$.

    Indeed, put $\nu = v_p(m_p')$. Then the Newton polygon of any $g= m! \cdot f$ with $f\in [f_{n,m}]$ lies above the polygon whose vertices are $\{(0,\nu+1), (m_p',1), (m_p,0), (k,0)\}$. Taking $j_1=m_p'$ and $j_2=m_p$ and $f$ as in the statement, we get that the Newton polygon contains the line whose end points are $\{(m_p',1),(m_p,0)\}$ (here we use that since $m_p'-n<p$, the points coming before $(m_p',1)$ lie above the line with slope $-1/p$ running through $(m_p',1)$). So $f$ has a root whose valuation is $-1/p$, hence the ramification index is divisible by $p$, hence the order of the Galois group is divisible by $p$, as claimed. 
    
    \begin{figure}[h]
        \centering
        \fbox{\includegraphics[width=8cm]{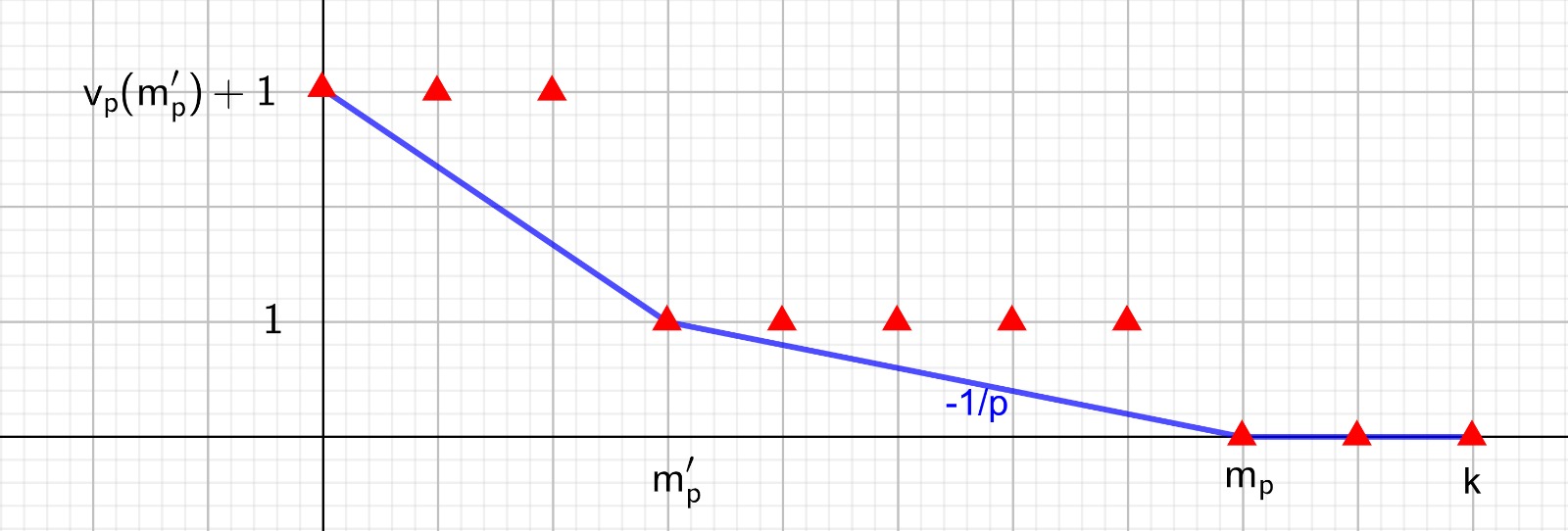}}
        \caption{The Newton polygon of $m!f_{n,m}$}
        \label{fig:NP_in_pf_GG}
    \end{figure}
    
    Take now $h=3k^{\rho}$ with $\rho = 0.525$. 
    By (\ref{eq-exist-prime-SI}), there exist three primes $\frac{k}{2} < p_1<p_2<p_3<\frac{k}{2}+h$. Since $\bigcup_{k/2<p<k-2} \Gamma_k^p$ contains the event in (\ref{eq-threeprimesprob}), it follows that 
    \[
        \mathbb{P}( \bigcap_{k/2<p<k-2}\neg \Gamma_k^p)  = O\Big(\frac{h^3}{k^3}+\frac{k^6}{x}\Big)= O(k^{-1-\epsilon}).
    \]
    By the union bound we have
    \[
        \mathbb{P}\Big(\bigcup_{\kappa_1\leq k \leq \kappa_2} \bigcap_{\frac{k}{2}<p<k-2} \neg \Gamma_p^k \Big) = O\Big( \sum_{\kappa_1<k<\kappa_2}k^{-1-\epsilon}\Big) = O(\kappa_1^{-\epsilon}),
    \]
    as needed.
\end{proof}

Next we deal with small $k$-s.

\begin{prop}
\label{RE-prop-Ak-small}
Let $\kappa_{1}$ and $t$ be function of $x$ such that for every $8\leq k\leq \kappa_1$ the interval $(k-t,k-2)$ contains a prime. For positive integer $k$, and two primes $q,p$, let $\Delta_k^{q,p}$ be the event that  for every $f\in [f_{n,m}]$ with $q\nmid a_{0}a_{p}$ we have $p\mid |G_{f}|$. Then
\[
    \mathbb{P} \Big(\bigcap_{8\leq k\leq \kappa_1} \bigcup_{\frac{k}{2}<p<k-2}\bigcup_{q>k} \Delta_k^{q,p} \Big) =1 -O\Big(\frac{t\log \kappa_1}{\log x} + \frac{t}{\kappa_1} \Big).
\]
\end{prop}

\begin{proof}
By Corollary~\ref{SM-cor-unbounded-non-square-probability}, with the required probability, there exist prime numbers $q_0,\ldots, q_{t-1}$ all bigger than $\kappa_1$ such that $v_{q_i}(m-i)=1$. It suffices to show that this implies the event in the assertion. 

Take $p$ to be the largest prime in the interval $(k-t,k-2)$. By Bertrand's postulate, $p>\frac k2$. Let $i=k-p$, so $i<t$. The Newton polygon $\NP^{(q_i)}(f)$ of any $f\in [f_{n,m}]$ with $q_i\nmid a_0a_p$ contains the segment whose endpoints are $\{(0,c+1), (p,c)\}$ for some $c$. Its slope is $-\frac 1p$, which means that one of the roots has valuation $1/p$, so the ramification index is a multiple of $p$, and so $p\mid |G_{f}|$, as needed. 
\end{proof}

\begin{proof}[Proof of Proposition~\ref{MAIN-prop-Ak-probability}]
    Take $\kappa_1 = \log x$, $\kappa_2^{7+\epsilon} = O(x)$ and $t=(\log x)^{0.525}$. If $x$ is sufficiently large, then there always  exists a prime in $(k-t,k-2)$ either by Bertrand postulate if $k$ is small or by (\ref{eq-exist-prime-SI}) if $k$ is large. 
    We apply Propositions~\ref{prop-largek-Ak} and \ref{RE-prop-Ak-small}, and we get that with probability $1-o(1)$, for every $8\leq k\leq \kappa_2$, there exists a prime $\frac{k}{2}<p<k-2$, a prime $q=p$ or $q>k$ and two indices $j_1, j_2=j_1+p$  such that for every $f\in [f_{n,m}]$ with $q\nmid a_{j_1}a_{j_2}$ we have that $p\mid |G_{f}|$. Since we condition on $f$ being irreducible, $G_{f}$ is transitive. Thus $G$ is primitive (cf.\ \cite[Lemma 2.4]{bary2009dirichlet}) and hence by Jordan's theorem \cite{jordan1873limite}, $G \geq A_{k}$.
\end{proof}

\subsection{Extending the Results to Small $k$}
\label{SubSec_small_k}
The results can be extended to $k\leq7$, $k\neq 6$. 

\begin{prop}
\label{RE-prop-k-leq7-neq6}
Choose $m$ uniformly at random in $\{ 1,\ldots,x\}$. Then
\[
\mathbb{P}\left(\forall\ 0 < k \leq 7,\ k\neq 6:G_{f_{n,m}}=S_{k}\right)=1-O\left(\frac{\log \log x}{\log x}\right).
\]
\end{prop}

\begin{proof}
Applying Corollary \ref{SM-cor-unbounded-non-square-probability}, with $k=\log x$ and $t=7$, shows that with probability
\[
1-O\left(\frac{\log \log x}{\log x}\right)
\] 
for every $0 \leq i < k$ there exists a prime $p_{i} > 7$ with
\[
v_{p_{i}}(m-i)=1.
\]
The Newton polygons at these primes give a $k$ cycle and a $k-1$ cycle in $G_{f_{n,m}}$. 
In particular $G_{f_{n,m}}$ is 2-transitive, and not contained in $A_{k}$.

Looking on the Newton polygon at the other primes, gives that $k-i$ divides $|G_{f_{n,m}}|$ for $i=0,\ldots, k-1$, hence  $\lcm(1,\ldots,k)\mid |G_{f_{n,m}}|$. 

By direct computation, the only subgroups of $S_k$ ($k\leq 7$, $k\neq 6$) satisfying these conditions are $S_k$ themselves. 
\end{proof}

\begin{rem}
Surprisingly, the proof above fails for $S_6$. 

Indeed, $G=\mathrm{PGL}_2(\mathbb{F}_5)\simeq S_5$ is a $2$-transitive subgroup of $S_{6}$, not contained in $A_{6}$, and of order $5!=120$ which is divisible by $60=\lcm(1,\ldots,6)$. 

Moreover, for every $1 \leq r \leq 6$, there is an element $g_{r} \in G$, whose decomposition to disjoint cycles contains an $r$-cycle.
\end{rem}
\clearpage

\bibliographystyle{plain}

\clearpage

\end{document}